\documentclass[11pt]{article}

\usepackage{amsmath,amssymb,amsthm}
\usepackage{dsfont}

\newcommand{\eat}[1]{ }

\def\ie{{\it i.e.},~}
\def\eg{{\it e.g.},~}
\def\etal{{\it et al.}~}
\def\TR{{\mathsf{T}}}

\def\E{\mathbb{E}}

\def\one{\mathds{1}}

\def\Poisson{\mathrm{Poisson}}
\def\Var{\mathrm{Var}}
\def\Binomial{\mathrm{Binomial}}
\def\event{{\mathcal E}}

\def\eps{\epsilon}

\def\G{{\mathcal G}}
\def\T{{\mathcal T}}

\newcommand{\Binom}{\operatorname{Binomial}}
\newcommand{\Pois}{\operatorname{Poisson}}

\newtheorem{theorem}{Theorem}
\newtheorem{question}{Question}
\newtheorem{lemma}{Lemma}
\newtheorem{definition}{Definition}
\newtheorem{remark}{Remark}

\theoremstyle{plain}
\newtheorem{proposition}{Proposition}
\newtheorem{conjecture}{Conjecture}
\theoremstyle{definition}
\newtheorem{algorithm}{Algorithm}
\usepackage{fullpage}
\usepackage{url}
\usepackage{hyperref}
\usepackage[usenames]{color}
\usepackage{graphicx}
\usepackage[font=small, labelfont=bf]{caption}
\usepackage{subcaption}
\usepackage{parskip}
\usepackage[normalem]{ulem}

\begin{document}

\title{Global and Local Information in Clustering Labeled Block Models}
\author{Varun Kanade\thanks{University of California, Berkeley. This author is
supported by a Simons Postdoctoral Fellowship. Email:
\url{vkanade@eecs.berkeley.edu}}
\and
Elchanan Mossel\thanks{University of California, Berkeley. This author
acknowledges the support of NSF (grants DMS 1106999 and CCF 1320105) and ONR
(DOD ONR grant N000141110140) Email: \url{mossel@stat.berkeley.edu}}
\and
Tselil Schramm\thanks{University of California, Berkeley. This material is
based upon work supported by a Berkeley Chancellor's Fellowship and the National
Science Foundation Graduate Research Fellowship Program under Grant No. DGE
1106400. Email:
\url{tschramm@cs.berkeley.edu}}}

\maketitle

\begin{abstract}
The stochastic block model is a classical cluster-exhibiting random graph model
that has been widely studied in statistics, physics and computer science. In
its simplest form, the model is a random graph with two equal-sized clusters,
with intra-cluster edge probability $p$, and inter-cluster edge probability
$q$. We focus on the sparse case, \ie $p, q = O(1/n)$, which is practically
more relevant and also mathematically more challenging. A conjecture of
Decelle, Krzakala, Moore and Zdeborov\'{a}, based on ideas from statistical
physics, predicted a specific threshold for clustering. The negative direction
of the conjecture was proved by Mossel, Neeman and Sly (2012), and more
recently the positive direction was proven independently by Massouli\'{e} and
Mossel, Neeman, and Sly.

In many real network clustering problems, nodes contain information as well. We
study the interplay between node and network information in clustering by
studying a \emph{labeled} block model, where in addition to the edge
information, the true cluster labels of a small fraction of the nodes are
revealed. In the case of two clusters, we show that below the threshold, a
small amount of node information does not affect recovery. On the other hand,
we show that for any small amount of information efficient local clustering is
achievable as long as the number of clusters is sufficiently large (as a
function of the amount of revealed information).

%
%Our results further show that even a vanishing fraction of labeled nodes
%allows one to break various algorithmic symmetries that exist in the unlabeled
%block model. In particular, it allows efficient recovery and identification of
%true cluster labels using a local algorithm.
\end{abstract}

\newpage

\section{Introduction}
\label{sec:intro}
%--------------------------------------------------------------------------
The stochastic block model is one of the most popular models for networks with
clusters.  The model has been extensively studied in statistics~\cite{HLL:1983,
SN:1997, BC:2009}, computer science (where it is called the planted partition
problem)~\cite{DF:1989,JS:1998,CK:2001,McSherry:2001} and theoretical
statistical physics~\cite{DKMZ:2011, ZKRZ:2012, DKMZ:2011a}.

The simplest block model has $k$ clusters of equal size, and is generated as
follows. Starting with $n$ nodes, each node $v$ is randomly assigned a label
$\sigma_v$ from  the set $\{1, \ldots, k\}$. For each pair of nodes, $(u, v)$,
if their labels are identical an edge is added between them with probability
$p$, otherwise an edge is added with probability $q$.  Often the case when $p >
q$ is considered, and the question of interest is understanding how large $p -
q$ must be for correct clusters recovery to be possible.  In the recovery
problem the input consists of the unlabeled graph and the desired output is a
partition of the graph.

Real world networks are typically sparse. Thus, an interesting setting in the
block model is when $p$ and $q$ are in $O(1/n)$.  Here, it is more convenient
to parametrize the problem by setting $p = a/n$ and $q = b/n$, where $a,b$ are
constants.  In the sparse setting, exact recovery is impossible as the
resulting graph will have isolated nodes. Moreover, it is easy to see that even
nodes with constant degree cannot be classified accurately given all other
nodes in the graph. Thus the goal is to find a partition that has non-trivial
correlation with the original clusters (up to permutation of cluster labels).
This has sometimes been referred to as the {\em cluster detection} problem (see
e.g.  \cite{DKMZ:2011}); throughout the paper we refer to it as the {\em
cluster recovery} problem (though note that the goal is not to recover every
cluster with probability 1).

General results of Coja-Oghlan~\cite{Coja-Oghlan:2010} imply that it is
possible to identify a partition that is correlated with the true hidden
partition when $(a - b)^2 \geq C k^4 (a + (k-1) b)$.  A beautiful physics paper
by Decelle \etal\cite{DKMZ:2011} conjectured that the recovery problem is
feasible for the case of two clusters when $(a - b)^2 > 2(a + b)$ and
impossible when $(a -b)^2 < 2(a +b)$. The non-reconstructability in the case
where $(a -b)^2 < 2(a + b)$ was proved by Mossel, Neeman and
Sly~\cite{MNS:2012}, and more recently the same authors~\cite{MNS:2013b} and
Massouli\'{e}~\cite{Massoulie:2013} independently showed that recovery is
possible when $(a - b)^2 > 2(a + b)$.

\subsection{The labeled stochastic block model}

The aforementioned results along with previous results for denser block models
provide a detailed picture of recovery in the stochastic block model.  However,
the model they consider is idealized and does not capture many aspects of real
network problems. One such aspect is that in many realistic settings, node
label information is available for some of the nodes.  For example, in social
networks, the group label of some individuals (nodes) is known. In metabolic
networks, the function of some of the nodes may be known.  Indeed, there has
been much recent work in the machine learning and applied networks communities
on combining node and network information (see for
example~\cite{ChWeSc:02,BaBaMo:02,BaBiMo:04}).  There are several ways in which
node and edge information can be incorporated; in real applications nodes and
edges contain rich information which is noisy, but correlated with the node's
``true" label and with the ``similarity" of pairs of nodes.

In this paper, we study a simple model which incorporates both node and edge
information which we call the \emph{labeled} stochastic block model. This
model has been considered previously in the physics
literature~\cite{DKMZ:2011,vSMGE:2013,AvSG:2010}. In addition to having the
unlabeled graph as an input, a {\em small} random fraction of the nodes' labels
are also provided as input to the clustering algorithm.

\subsection{The big effect of a small number of node labels}

It is easy to see that even a vanishing fraction of node labels can play a major
role in the cluster recovery problem. For example, consider the denser case
where the clusters $C_1,\ldots,C_k$ can be identified
accurately~\cite{McSherry:2001}. Here, it is impossible to distinguish between
a clustering $C_1,\ldots,C_k$ where the nodes in cluster $C_i$ have label $i$
and the same clustering where the nodes in cluster $i$ have label $\pi(i)$ for
any permutation $\pi$ of the labels. However, note that for any $p>0$, given a
$p$-fraction of the node labels, it is possible to identify the permutation
$\pi$ correctly with high probability. It is natural to ask if the same result
holds in the sparse case, and it is not hard to see that a similar statement
can be made (see Proposition \ref{prop:permutation}).

The above observation shows that even a small amount of node information can
overcome the problems of symmetry in the stochastic block model. Another
problem of symmetry present in the unlabeled model is that there is no {\em
local algorithm} that can identify clusters better than random guessing.
Informally, a local algorithm determines the label of a node based solely on an
$o(\log n)$ neighborhood of that node, including possibly uniform independent
random variables attached to each node of the graph (see \ref{app:local} for a
formal definition and ~\cite{LyonsNazarov:11,HaLoSz:12} for examples). The
proof that a local algorithm cannot detect better than random guessing in this
case is folklore, and we include it here for completeness.  This limitation in
detection may be compared to the problem of finding independent sets, where
local algorithms can have non-trivial power (while still being less powerful
than global algorithms)~\cite{GamarnikSudan:14}.  It is therefore natural to
ask:

\medskip
\begin{question}
Does a vanishing fraction of labeled nodes allow {\em local algorithms} to
detect clusters? If so, when?
\end{question}

\medskip
An even a more direct question relates to the statistical power of revealing
some of the node labels. While it is clear that revealing a large fraction of
the node labels allows non-trivial recovery, it is far from clear what the
effect is when this fraction is vanishingly small. On the one hand, we might
expect by continuity that revealing a vanishing fraction of the node labels
will be identical in the limit to revealing no labels.  On the other hand, we
might imagine how a small fraction of the node labels could be used as  seeds
for recovery algorithms. We thus ask:

\medskip
\begin{question}
Does revealing a vanishing fraction of the node labels change the detectability
threshold? Does it change the fraction of correctly labeled nodes?
\end{question}
\medskip

The latter question was considered in recent work in statistical
physics~\cite{Moore:2013-misc,vSMGE:2013,AvSG:2010}.

\subsection{Our results}

To set the stage for our contributions, we begin with some observations
regarding the utility of local information. The proofs of these propositions are
straightforward (see Appendix~\ref{sec:appendix}), but they are useful for
establishing context of how information about (a small fraction of) node labels
may help. The first is that even a vanishingly small proportion of node labels aids
in breaking the symmetry and assigning labels to the cluster assignments. \medskip

%-----------------------------------------------------------
\begin{figure}[t]
\begin{center}
	\includegraphics[width=\textwidth]{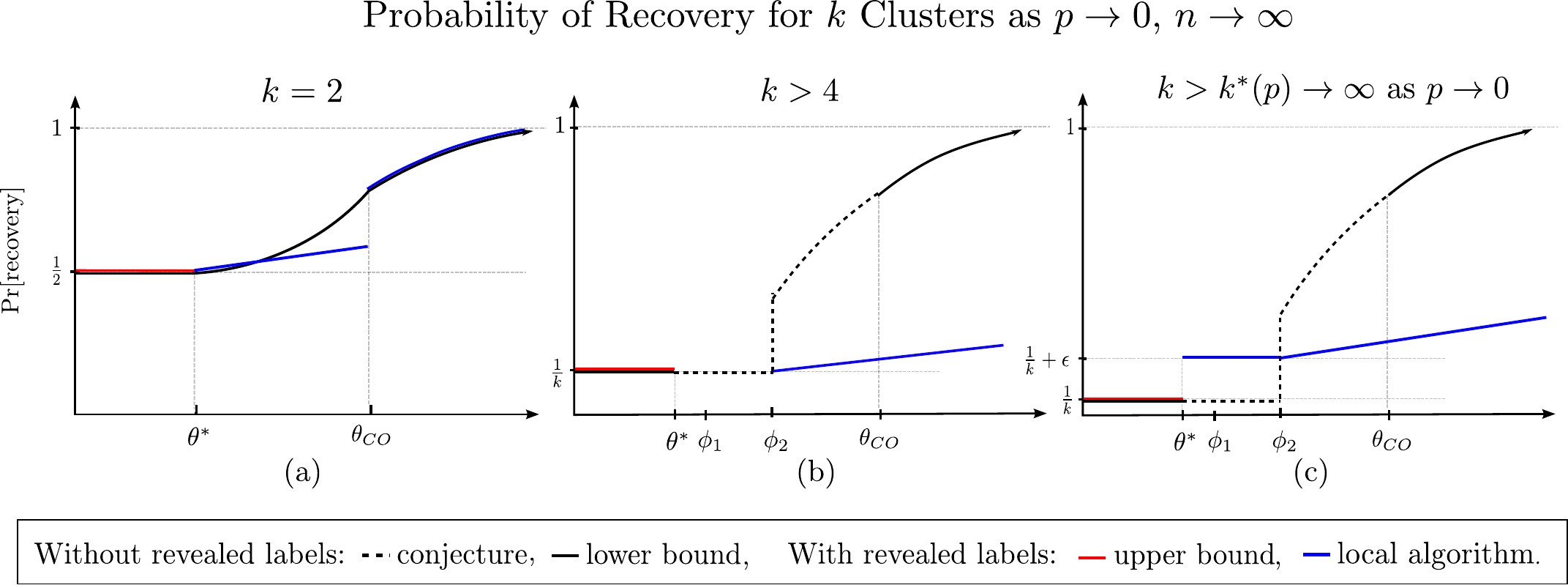}
	\end{center}
	\caption{
Previous work (black) and our contributions (colored).  The $x$-axes
represent the second eigenvalue of the corresponding broadcast process on the
coupled Galton-Watson Tree when the average degree is fixed--in simpler terms,
this is an increasing function of the ratio $\tfrac{a-b}{a}$.
	In all three cases, $\theta^*$ is the reconstruction threshold
	corresponding to the root reconstruction problem on trees, and
	$\theta_{CO}$ is the threshold of \cite{Coja-Oghlan:2010}.  In the
	two-cluster case (Subfigure (a)), $\theta^*$ corresponds exactly to the
	Kesten-Stigum bound of $(a-b)^2 < 2(a+b)$ \cite{DKMZ:2011, MNS:2012}. For
	the case of larger $k$, $\theta^* < (a-b)^2/k(a+(k-1)b)$ (see Subfigures
	(b), (c) and Proposition \ref{prop:k_threshold}). We prove analogously that
	recovery is not possible below $\theta^*$ in the labeled model as $p\to0$
	for all $k$ (Theorem \ref{thm:k_clusters}). In the two-cluster case, recent
	results of \cite{MNS:2013b} and \cite{Massoulie:2013} show that recovery is
	possible in the range $(\theta^*, \theta_{CO})$; above $\theta_{CO}$, a
	combination of the results of \cite{Coja-Oghlan:2010} and \cite{MNS:2013a}
	give optimal recovery in the standard model for $k =2$; we observe that in
	the labeled model for $k=2$, one can reconstruct better than randomly in
	the range ($\theta^*$, $\theta_{CO}$) and optimally above $\theta_{CO}$
	using {\em local} algorithms (see
	Propositions~\ref{prop:local_count}~and~\ref{prop:local_with_info}).  The
	results of \cite{Coja-Oghlan:2010} also give non-trivial recovery
	guarantees above $\theta_{CO}$ for all $k$.  In the $k$-cluster case
	(Figures (b), (c)), the picture is more complicated: $\phi_1$ and $\phi_2$
	are conjectured brute-force and efficient solvability thresholds
	respectively, both conjectured by \cite{DKMZ:2011}---above $\phi_1$
	recovery is possible via brute-force enumeration, and above $\phi_2$ an
	efficient algorithm for recovery exists. Above $\phi_2$,
	Proposition~\ref{prop:local_count} shows that recovery is possible for $k$
	clusters via a {\em local} algorithm.  In Subfigure (c), for any $b,p$, if
	$k > k^*(p)$ and $(a-b)/k > 1$, as in Theorem \ref{thm:reconstruction},  we
	give an efficient {\em local} recovery algorithm that correctly labels
	$\frac{1}{k} + \epsilon$ of the nodes, even below the conjectured efficient
	recovery threshold $\phi_2$. \label{fig:two}}
\end{figure}

\begin{proposition}[Informal version] \label{prop:permutation_informal}
Given a clustering algorithm which outputs clusters correlated with the true
clustering, a small fraction of revealed node labels is sufficient to output a
labeling which is correlated with the true labeling.
\end{proposition}

In the absence of any node information, it is an easy folklore result that any
local algorithm cannot recover clusters. However, we show that in the case of
two clusters, when a small fraction of node labels are revealed, a local
algorithm is able to recover the clusters optimally.  This latter result is a
direct corollary of a robust reconstruction result on trees of~\cite{MNS:2013a}.\medskip
\begin{proposition}[Informal version] \label{prop:local_sucks_informal} In the
unlabeled stochastic block model, no local algorithm can find a clustering
correlated with the true clustering.
\end{proposition} \smallskip

\begin{proposition}[Informal version] \label{prop:local_with_info_informal}
In an instance of the labeled stochastic block model, when $k = 2$, if $(a -
b)^2 > C (a + b)$ for some large constant $C$, then there is a local algorithm
which given a vanishing fraction of labeled nodes, reconstruct the label of
all nodes with the same accuracy as the optimal (non-local) algorithm for the
unlabeled problem.
\end{proposition}
\medskip

We also observe that results on census reconstruction~\cite{MosselPeres:03}
imply that above the Kesten-Stigum bound a vanishingly small fraction of
revealed nodes suffices for the cluster recovery problem.
\medskip

\begin{proposition}[Informal version] \label{thm:k_clusters_ks_informal}
For any fixed $k$,  above the robust reconstruction threshold (i.e. when
$(a-b)^2 > k(a +(k-1)b)$), when the fraction of revealed node labels is
vanishingly small, the cluster recovery problem is solvable by a local
algorithm.
\end{proposition}
The proof follows more or less directly from previous results, but we include it in
Appendix~\ref{sec:recon} for completeness.

\medskip

%--------------------------------------------------

In this context, one might expect that labels could allow clustering in the
labeled model in regimes which cannot be effectively clustered in the unlabeled
model. The case of two clusters is the case we understand the best. Here,
utilizing results for the reconstruction problem on trees and of
\cite{MNS:2012}, we answer Question 2 in the {\em negative} (Theorem
\ref{thm:k_clusters_informal}) and at the same time answer Question 1
positively (Propositions \ref{prop:local_with_info} and
\ref{prop:local_count}). The complete picture for the case of two clusters is
presented in Figure~\ref{fig:two}(a).

For any fixed $k > 2$, the picture is much more complicated. In this case, we
observe that below the tree reconstruction threshold (this corresponds to
$\theta^*$ in Figure~\ref{fig:two}(b)), a vanishing fraction of node labels do
not assist in the cluster recovery problem (see Theorem
\ref{thm:k_clusters_informal}). \medskip

\setcounter{theorem}{1}
\begin{theorem}[Informal version] \label{thm:k_clusters_informal}
For any fixed $k$, below the  associated tree reconstruction threshold (to be
defined later), when the fraction of revealed node labels is vanishingly small,
the cluster recovery problem is not solvable. In particular, when $k = 2$, the
threshold is the Kesten-Stigum bound of $(a - b)^2 < 2 (a + b)$; for $k \geq 2$,
if $a - b < k$ then recovery is impossible.
\end{theorem}
\medskip

%--------------------------------------------------

\setcounter{theorem}{0}
Our main interest is in the case when the number of clusters is very large.
Here, we consider the setting when the fraction of revealed nodes $p \to 0$,
and simultaneously the number of clusters $k = k(p) \rightarrow \infty$. In
this setting, we show that revealing node labels has a dramatic effect on the
threshold for cluster recovery. We show that a local algorithm successfully
solves the cluster recovery problem even below the conjectured algorithmic
threshold in the unlabeled case, $(a - b)^2 = k(a + (k-1) b)$. As the number of
clusters $k \to \infty$, our algorithm works all the way down to the tree
reconstruction threshold of $(a - b)/k > 1$. Moreover, it is impossible to
recover (locally or globally) with a vanishing fraction of labeled nodes if
$(a-b)/k < 1$. Both results follow from the corresponding results on trees.

\medskip
\setcounter{theorem}{0}
\begin{theorem}[Informal version]
\label{thm:reconstruction_informal}
For every $\delta >0$, there exists $\epsilon = \epsilon(\delta) > 0$ such that
for every $p > 0$, if $k = k(p)$ is large enough as a function of $p$ and $a -
b > (1 + \delta)k$, then the label of a random node can be recovered with
probability at least $\tfrac{1}{k} + \epsilon$.
\end{theorem}
Note that $\epsilon$ depends on $\delta$ but is independent of $p$.

Recent work in statistical physics~\cite{Moore:2013-misc} argues that for every
\emph{fixed number of clusters} $k$, a vanishing fraction of labels does not
provide any advantage in the detection probability over having no labels at
all. We note that in our results, the order of limits is exchanged as the number
of clusters $k$ needed for our results to hold, depends on the fraction of
nodes revealed. Thus, there is no contradiction between the results (see
also~\cite{AvSG:2010,vSMGE:2013}). Figure~\ref{fig:two}(c) provides a detailed
picture of the case in which the number of clusters is very large (in the setting of Theorem~\ref{thm:reconstruction_informal}).

\subsection*{Open Problems}

In the case of two clusters, we conjecture that whenever any fraction of node
labels are revealed, there is a \emph{local algorithm} that recovers the
clusters optimally. This would follow from a related conjecture regarding
information flow on trees stated below. We report some simulations suggesting
the veracity of the conjecture in Appendix~\ref{sec:expts}. \medskip

\begin{conjecture}[Informal version] \label{conj:informal}
Let $T$ be an infinite tree with root
$\rho$. The tree is labeled from the set $\{\pm 1\}$ as follows. First, the root
is assigned a label from $\{\pm 1\}$ at random. Along each edge
the label is propagated with probability $1 - \eta$ and flipped with
probability $\eta$. Let $(T, \tau)$ denote the resulting labeled tree. Add
each node independently to a set $R$ with probability $p$. Finally for any $r$,
let $\partial T_r$ denote the set of leaves at depth $r$. Then, for any value
of $p > 0$ and $\eta
< 1/2$,
\[
\lim_{r \to \infty} \E \big|\Pr[ \tau_{\rho} = 1 ~|~ \tau_{R}] -
\Pr[\tau_{\rho} = 1~|~ \tau_{R},\tau_{\partial T_r}]\big| = 0
\]
\end{conjecture}

In addition to Conjecture \ref{conj:informal}, several interesting questions
remain, particularly in the regime where $k$ is large. When $k$ is large, is it
possible to use global and local information together to obtain better
recovery guarantees? Which algorithmic tools might allow one to use global
and local information simultaneously?

Another open problem relates to different types noise models. The assumption in
the current paper is that each label is revealed accurately with a vanishing
probability. But one may consider other types of noise. In particular, we may
assume for example that for each node independently we are given the correct
label with small probability $\delta$ and otherwise a uniformly chosen label. Is
it true that the same results hold for this noise model as for the noise model
considered here? For most of the results presented here, it is easy to see that
the answer is yes. However, for one of our main results,
Theorem~\ref{thm:reconstruction_informal}, the proof {\em does not} extend to
the latter noise model. It is an interesting open problem to determine the
effect of the noisy information in this setup.

\medskip
\begin{remark}
A short abstract describing these results will appear in proceedings of RANDOM 2014.
\end{remark}

\subsubsection*{Acknowledgments}

E.M. thanks Cris Moore, Joe Neeman, Allan Sly and Lenka Zdeborov\'{a} for many
interesting discussions related to the block model. We would like to thank the
authors of ~\cite{Moore:2013-misc} for discussion of their work at its early
stages.  The authors would like to thank the Simons Institute for the Theory of
Computing where much of the work reported here was carried out. The authors
would also like to thank anonymous referees for their helpful comments.

\setcounter{theorem}{0}

\section{Model}
\label{sec:model}
\subsection{Stochastic Block Model}

The stochastic block model is a generative model for modular random networks,
defined by the following set of parameters: the number of clusters $k$, the
expected fraction of nodes in each cluster $i$, $\langle f_i \rangle_{i=1}^k$ ,
and a $k \times k$ symmetric affinity matrix $P_{i, j}$ indicating the edge
probability between nodes of type $i$ and $j$. A random network $G$ on $n$ nodes
is generated as follows:
\begin{enumerate}
\item First, each node $v$ is assigned a label $\sigma_v \in \{1, \ldots, k\}$, s.t.
$\Pr[\sigma_v = i] = f_i$.
\item For every pair of nodes $u, v$, an edge is added between them with
probability $P_{\sigma_u, \sigma_v}$, independently for each pair.
\end{enumerate}
In this work, we are mainly interested in the sparse case, \ie when the average
degree of the graph is constant. We focus on the setting where edge
probabilities only depend on whether the labels of the endpoint are same or
different. Thus, $P_{ii} = a/n$ for $1 \leq i \leq k$ and $P_{ij} = b/n$ for $i
\neq j$, for constants $a > b$.\footnote{This is the so-called assortative
model.} Also, we focus on the case where $f_i = 1/k$ for each $i$, \ie each
cluster is roughly of the same size. The model is denoted by $\G(n, k, a, b)$,
and $(G, \sigma) \sim \G(n, k, a, b)$ denotes an instance of a graph generated
according to the model, where $\sigma$ are the cluster labels of the nodes.
\smallskip

\noindent{\bf Labeled Block Model}: The labeled block model has an additional
parameter $p$, which is the probability with which the true cluster label of any
given node is revealed. Thus, if $(G, \sigma) \sim \G(n, k, a, b)$ is an
instance of the block model, $R \subseteq [n]$ is chosen by placing each node of
$G$ in $R$ independently with probability $p$. We denote this by $(G, \sigma, R)
\sim \G(n, k, a, b, p)$. The clustering algorithm has access to the edges of
$G$ and the cluster labels $\sigma_R$ of nodes in $R$, \ie $(G, R, \sigma_{R})$. 

We also introduce the following notation for convenience. For any two nodes
$u,v \in G$, let $d(u,v)$ denote the distance between $u$ and $v$. We let
$G_r(v) = \{u \in G ~|~ d(u,v) \le r\}$ denote the neighborhood of radius $r$
around $v$; at times we will use $G_r$ when $v$ is clear from context. Let
$\partial G_r(v) = \{u \in G ~|~d(u,v) = r\}$ denote the boundary of $G_r(v)$. \smallskip

\noindent{\bf Cluster Recovery}: The {\em cluster recovery} problem is the
problem of recovering the cluster label of nodes in the stochastic block model
or labeled stochastic block model with better-than-random probability. Note
that correct recovery of all nodes is not the aim, nor is it possible due to
the sparsity of the graph. This problem has also been called the {\em cluster
detection} problem and the {\em cluster reconstruction} problem; for
consistency we will use the term recovery throughout the paper when referring
to graphs, and use reconstruction when referring to broadcast processes on trees.

\subsection{Information Flow on Trees}
\label{sec:model-trees}

We use some results regarding information flow on trees. For a detailed survey
on this topic, the reader is referred to~\cite{Mossel:2004}.

Let $T$ be an infinite rooted tree, with the root note denoted by $\rho$. A
Galton-Watson tree is obtained by starting with a root node, $\rho$, and
recursively adding offspring drawn from some distribution $D$ with mean $d$. In
particular, we will often be interested in the case when $D$ is $\Poisson(d)$.
For any node $v \in T$, let $d(v, \rho)$ denote the distance of $v$ from the
root.  Throughout the paper, we denote $T_{r} = \{v \in T ~|~ d(v, \rho) \leq r
\}$ as the subtree of $T$ up to depth $r$, and $\partial T_{r} = \{ v \in T ~|~
d(v, \rho) = r \}$ as the boundary at depth $r$. \smallskip 

\noindent{\bf Broadcast Process}: Let $T$ be an infinite rooted tree with root
$\rho$. Each node in the tree is assigned a label from some finite alphabet
$\Sigma = \{1 , \ldots, k\}$. The root is labeled by choosing a label
$\tau_{\rho} \in \Sigma$ uniformly at random. For any edge $(u, v)$, with $d(u,
\rho) < d(v, \rho)$, $\tau_v$ is conditionally independent given $\tau_u$, and
is chosen as follows: $\tau_v = \tau_u$ with probability $1 - (k -1) \eta$, and
$\tau_v \in \Sigma \setminus \{\tau_u\}$ randomly otherwise,
where $\eta < 1/k$ is the broadcast parameter.  We denote this process by
$\T(T, k, \eta)$ and an instance generated according to this process by $(T,
\tau) \sim \T(T, k, \eta)$. As in the block model, we can consider the process
when the label of each node is revealed with probability $p$, \ie $R \subseteq
T$ is obtained by adding each $v \in T$ to $R$ independently with probability
$p$. We denote this process by $(T, \tau, R) \sim \T(T, k, \eta,p)$. The
{\em reconstruction problem} is to identify the label of the root, $\rho$ given the
labeled nodes up to some depth $r$. Thus,  the algorithm has access to
$(T_{r}, R_{r}, \tau_{R_{r}})$, where $R_{r}$ denotes $T_{r}
\cap R$.  \smallskip

\noindent{\bf Percolation Process}: Let $T$ be an infinite rooted tree with root
$\rho$. For percolation parameter $\lambda$, each edge $e \in T$ is deleted
independently with probability $\lambda$. Let $C(\rho)$ denote the component of
$T$ containing the root after percolation.

\section{Recovery in the many clusters regime}
\label{sec:many_clusters}
We show that when the number of clusters is very large, even a very small
fraction of revealed node labels allow for cluster recovery, and even in some
regimes below the conjectured algorithmic threshold in the standard model.
More formally, if $p$ is the probability that the label of a node is revealed,
and if the number of clusters is at least $k^* = k(p)$, then even as $p
\rightarrow 0$, the algorithm performs better than random assignment.  The
algorithm (Algorithm~\ref{alg:many_clusters}) is simple and \emph{local}---it
considers a neighborhood around each node and uses the revealed node
information in the neighborhood to make its prediction.
%---------------------------------------------------------
\begin{figure}[t]
\begin{center}
\fbox{
\begin{minipage}{0.85 \textwidth}
\begin{algorithm}
\mbox{ }

\noindent{\bf Input:} $(G,R) \sim \G(n, k, a, b, p)$, radius $r$, max-degree $D$,
revealed cluster labels $\sigma_R$ \medskip

For each node $v \not\in R$
\begin{enumerate}
\item Let $G_r(v)$ denote the (tree-like) neighborhood of $v$ up to distance $r$
\item From $G_r(v)$ delete every subtree rooted at a node with degree larger than $D$
\item Let $L$ denote the set of labels $l \in \Sigma$ for which there
exist $x,y \in R$ such that $\sigma_x = \sigma_y = l$, $d(x,v) = d(y,v) = r$, and
$v$ is $x$ and $y$'s first common ancestor
\item Assign a random label from $L$ to node $v$
\end{enumerate}
\label{alg:many_clusters}
\end{algorithm}
\end{minipage}
}
\end{center}
\end{figure}
%---------------------------------------------------------

\medskip
\begin{theorem}\label{thm:reconstruction}
Let $b > 1$ be fixed, let $a = b + (1 + \delta)k$ for some $\delta > 0$, let $p
> 0$ be fixed. Then, there exists an $\epsilon = \epsilon(b, \delta)$ and $k^* =
k^*(b, \delta, p)$, such that for every $k \geq k^*$, if $(G, R, \sigma_{R})
\sim \G(n, k, a, b, p)$, Algorithm~\ref{alg:many_clusters} labels any random node
of $G$ correctly with probability at least $\epsilon$. In particular,
there exists settings where $(a - b)^2 < k (a + (k-1)b)$ and recovery
is still possible.
\end{theorem}

Before we present a formal proof of Theorem~\ref{thm:reconstruction},
we give a high-level idea of the proof. First, we utilize a coupling between
local neighborhoods in $\G(n,k,a,b)$ and a broadcast process on a rooted
Galton-Watson tree with offspring distribution $\Poisson(\tfrac{a +
(k-1)b}{k})$.  Fix $v \in [n]$ and let $(G, \sigma) \sim \G(n, k, a, b)$. For
large values of $n$, and when $r$ is not too large (though increasing as a
function of $n$), $G_r(v)$ looks like a tree. The degree distribution of any
node in $G$ is $\Binomial(n, \tfrac{a + (k-1)b}{kn}) \approx \Poisson(\tfrac{a
+ (k- 1)b}{k})$. If $\eta = \tfrac{b}{a + (k-1)b}$, the distribution $(G_r,
\sigma_{G_r})$ resembles the distribution $(T_r, \tau_{r})$, where $(T, \tau)
\sim \T(T,k,\eta)$ corresponds to the broadcast process on a Galton-Watson tree
process $T$ with offspring distribution  $\Poisson(\tfrac{a + (k-1)b}{k})$.
This coupling was formally proved in~\cite{MNS:2012}. \medskip

\begin{lemma}[\cite{MNS:2012}] \label{lemma:tree-coupling}
\label{lem:treelike}
Let $r < r(n) = \frac{1}{10 \log(2(a + (k-1)b))} \log(n)$. There exists a
coupling between $(G, \sigma)$ and $(T, \tau)$ such that $(G_r, \sigma_{G_r}) =
(T_r, \tau_{T_r})$
a.a.s.
\end{lemma}

In~\cite{Mossel:2001} it is shown that for larger alphabet sizes, $d(1 - k
\eta)^2 \geq 1$ is not the threshold for reconstruction for regular trees. As
our results show, this is also the case for Galton-Watson trees.  In order to
understand the intuition behind Algorithm~\ref{alg:many_clusters}, it is useful
to consider an \emph{infinite color} broadcast process on a tree.  Let
$\tilde\eta \ll 1$ be a small broadcast parameter. Suppose the root $\rho$ is
given some color, which is propagated away from the root as follows.  With $(1
- \tilde\eta)$ probability the neighboring node gets the same color, with
$\tilde\eta$ probability the neighboring node gets a completely new color. The
color of each node is revealed with probability $p$. Consider the following
event: there are two nodes in the tree with the same color, for which the root
$\rho$ is the first common ancestor. If such an event occurs, this color
\emph{must} also be the color of the root. We show that this infinite-color
picture is more or less accurate when $k$ is large enough.

\medskip

We now prove Theorem~\ref{thm:reconstruction} through a sequence of lemmas.

Let $T$ be a Galton-Watson tree with offspring distribution $\Pois(d)$ for $d =
\tfrac{a+(k-1)b}{k}$, and let $\eta = \tfrac{b}{a + (k-1)b}$ be the parameter
of the $k$-label broadcast process on $T$ (so that $(T, \tau, R) \sim \T(T,k,
\eta, p)$). Consider the coupling between $(G,\sigma,R)$ and $(T,\tau,R)$ as
per Lemma~\ref{lemma:tree-coupling}.

Next, we relate the broadcast process on $T$ to a percolation process on $T$.
Suppose the root is labeled according to some $\tau_{\rho} \in \Sigma = \{1,
\ldots, k\}$. Then, across any edge the probability that the label remains
unchanged is $1 - (k - 1)\eta$. Thus, if we look at a percolation process with
$\lambda  = 1 - (k - 1)\eta$, then the connected component $C(\rho)$ corresponds
to a tree in which every node has the same label as the root.

\medskip
\begin{lemma}
\label{lem:min2}
Let $T$ be an infinite rooted tree with root $\rho$ and where the degree of each
node is chosen from a distribution with mean $d$.  Let $R \subseteq T$ be
obtained by adding each $v \in T$ to $R$ independently with probability $p$. Let
$\lambda$ be the percolation parameter such that $d \lambda > 1$.  Then in the
percolated tree, for any $B > 0$ there exist $\ell(d\lambda, B, p),
\epsilon(d\lambda)$ such that
\[
\Pr[|C(\rho) \cap \partial T_\ell \cap R| \ge B] \ge \epsilon.
\]
\end{lemma}
\begin{proof}
For any $\ell$, let $Z_\ell = C(\rho) \cap T_\ell$, and define $W_\ell =
(d\lambda)^{-\ell}|Z_\ell|$. Observe that $d\lambda > 1$, and
\[
\E[W_{\ell + 1} ~|~ W_\ell] = W_\ell,
\]
and so $W_\ell$ is a positive martingale. Therefore, $W_\ell \rightarrow W$
a.s. Moreover, since this is a branching process, it is known that when $d
\lambda > 1$, $\Pr[W \neq 0] = \lim_{\ell \rightarrow \infty} \Pr[Z_\ell \neq
0] > 0$~\cite{AN:1972}.  Therefore, there exist $\epsilon, \epsilon_1$ such
that
\begin{equation}\label{eq:Zlbound}
	\Pr[|Z_\ell| \ge \epsilon_1(d\lambda)^\ell] > 4\epsilon \text{ for all }
	\ell.
\end{equation}

Now, it remains to bound $|Z_\ell \cap R|$. Since each node in $T$ is in $R$
independently with probability $p$,
$|Z_\ell \cap R| \sim \Binom(|Z_\ell|, p)$.
We choose the smallest $\ell$ such that
$\epsilon_1(d\lambda)^{\ell} = m$ and $\Pr[\Binom(m,p) > B] > \frac{1}{4}$, so that
\begin{align*}
\Pr[\Binom(|Z_{\ell}|, p) \ge B]
&= \sum_{q = 0}^{\infty} \Pr[\Binom(q, p) \ge B ~|~ |Z_{\ell}| = q]\cdot \Pr[|Z_{\ell}| = q]\\
&\ge \Pr[\Binom(m, p) \ge B]\cdot \Pr[|Z_{\ell}| \ge m]\\
&\ge \epsilon,
\end{align*}
where the first inequality follows from independence and from the fact
that $\Pr[\Binom(q,p) \ge B]$ is increasing in $q$, and the second inequality
is an application of Equation \ref{eq:Zlbound}.

Thus, our conclusion follows using $\epsilon$ and $\ell$. Note that $\epsilon$
only depends on the product $d\lambda$, and $\ell$ depends on $d\lambda$, $p$
and $B$.
\end{proof}

\begin{lemma}
\label{lem:only1}
Let $T$ be an infinite rooted tree with root $\rho$ and maximum degree $D$, and
let $T$ be labeled according to the broadcast process with
$\Sigma = \{1, \ldots, k\}$ and $\eta < 1/k$. Let $A_{u, v}$ be the event that
two nodes $u$ and $v$
have $\rho$ as their first common ancestor.
Then for any $\epsilon, \ell$, there exists $k^*(D, \ell, \epsilon)$ such that for
all $k \geq k^*$,  for event $\event$ defined as
%%%
\[ \event : \exists u, v \in \partial T_{\ell + 1} ~s.t.~ A_{u, v}, \tau_{u} =
\tau_{v} \neq \tau_{\rho}, \]
then $\Pr[\event] \leq \epsilon$.
\end{lemma}
\begin{proof}
Say that a \emph{mutation} occurs if the color changes along any edge. We note
that in order for the event $\event$ to occur, two mutations must occur in the
subtrees corresponding to different children of $\rho$, since $\rho$ must be the
first common ancestor. By the Markov property of the broadcast process, it
follows that the two mutations must be independent. Hence, it suffices to bound
the probability of two independent mutations to the same color.

In $T_{\ell + 1}$, there are at most $D^{\ell + 1}$ edges. For any fixed color,
the probability that there is a mutation to that color along any edge is at
most $\eta D^{\ell + 1}$ by union bound, so the probability that there are two
independent mutations to that specific color is at most $\eta^2 D^{2 \ell +
2}$. Taking a union bound over all the colors, we observe that the probability
of the event is at most $k \eta^2 D^{2 \ell + 2}$. Thus, when $k^* \geq
\tfrac{D^{2 \ell + 2}}{\epsilon}$, for any $k \geq k^*$, the statement of the
Lemma holds.
\end{proof}
\medskip

Before proving Theorem \ref{thm:reconstruction}, we prove the corresponding
version for Galton-Watson trees.\\

\begin{proposition}\label{prop:tree-proof}
Let $T$ be a Galton-Watson tree with offspring distribution $\Poisson(d)$. Let
$p > 0$ be fixed. Then there exists $k^*, \epsilon$, such that for any $k \geq
k^*$, if $\eta \leq (d-1 - \delta)/kd$ for $(T,R,\tau)
\sim \T(T,k,\eta,p)$, then
given $(T_{\ell}, R \cap T_{\ell}, \tau_{R})$, the label of the root can be
reconstructed with probability at least $\epsilon$.
\end{proposition}
\begin{proof}
First, we check that $\lambda = 1- k\eta = \tfrac{1 + \delta}{d}$. Thus,
$\lambda d = 1 + \delta > 1$.

In order to apply Lemma~\ref{lem:only1}, it is necessary to bound the degree
of the tree by some $D$. In general, the degree of a Galton Watson tree with
offspring distribution $\Poisson(d)$ is not bounded. Instead, we consider a tree
with a modified, bounded degree distribution, $Y$.
Let $Y_0 \sim \Poisson(d)$, let $Y = Y_0$
if $Y_0 \leq D$, and $Y_0 = 0$ otherwise. Choose $D$ such that
$\sum_{i = D}^{\infty} i\tfrac{e^{-d}d^i}{i!} \leq \delta/2$.
Thus, $d^\prime = \E[Y] \geq d -
\delta/2$. Using the fact that $\lambda < 1$, we know that $d^{\prime} \lambda
\geq 1 + \delta/2$. Thus, given a Galton-Watson tree, we can first prune the
tree by deleting any node
that has degree strictly larger than $D$. Call this resulting tree $T^\prime$.

Consider the following event: The root $\rho$ has two children that are
retained in $T^\prime$ and have label $\tau_\rho$. The probability of this
event is at least
$\epsilon_1$, where $\epsilon_1$ depends only on $d$ and $\delta$. Assume that
this event has occurred and let $v_1$ and $v_2$ be these children. Now, we apply
Lemma~\ref{lem:min2} with $B = 1$ to both $v_1$ and $v_2$ to see that with
probability at least $\epsilon_2$ each of $v_1$, $v_2$ has a revealed
descendant at level $\ell(\epsilon_2)$ with label $\tau_\rho$.
Let $\event_{\text{good}}$ denote the event that there
exist two nodes $w_1$ and $w_2$ in $\partial T^\prime_{\ell + 1}$ with $\rho$ as
their first common ancestor and  $\tau_{w_1} = \tau_{w_2} = \tau_\rho$.
Then, $\Pr[\event_{\text{good}}] \geq \epsilon_1\epsilon_2^2$, since the subtrees rooted
at $v_1, v_2$ are conditionally independent.

Let $\ell$ be as obtained above and let $\epsilon = \epsilon_1 \epsilon_2^2/2$.
Now we appeal to Lemma~\ref{lem:only1}, to obtain a value of $k^*$, such that
for any $k \geq k^*$, $\Pr[\event_{\text{bad}}] \leq \epsilon$, where $\event_{\text{bad}}$ is the event defined in Lemma~\ref{lem:only1}. Thus, the algorithm that
looks for two nodes with the same label and having the root as the first common
ancestor, succeeds in labeling the root correctly with probability at least
$\epsilon$.
\end{proof}

Finally, we can appeal to Proposition~\ref{prop:tree-proof} to complete the
proof of Theorem~\ref{thm:reconstruction}.

\begin{proof}[Proof of Theorem~\ref{thm:reconstruction}]
By Lemma~\ref{lem:treelike}, for $(G,R,\sigma) \sim \G(n,k,a,b,p)$,
if $(T, R, \tau) \sim \T(T,k,\eta,p)$ where
$T$ is a Galton-Watson tree with offspring distribution $\Pois(d)$
where $d = \frac{a + (k-1)b}{k}$ and  $\eta = \frac{b}{a + (k-1)b}$,
then we can couple $G_r(v)$ with $T_r$. Note that $\lambda = 1 - k \eta$ is
equal to $(1 + \delta)/d$, and thus Proposition \ref{prop:tree-proof} implies the
desired result immediately.
\end{proof}

\section{Upper bounds below the threshold}
\label{sec:two_clusters}
In this section, we consider the setting where there are a fixed number of
clusters and the fraction of revealed node labels is vanishingly small. We show
that below a certain threshold that arises from the reconstruction problem on
trees, in the limit as $p \rightarrow 0$, cluster recovery is not
possible. We first note that a threshold exists for the tree problem. \medskip

\begin{proposition}\label{prop:k_threshold}
Let $T$ be a Galton-Watson tree with average degree $d > 1$. Let $(T,\tau) \sim
\T(T, k,\eta)$ be the labels obtained by the broadcast process with parameter
$\eta$. There there exists a predicate, $\pi_{k}(d, \eta)$, monotonically
decreasing in $\eta$ and monotonically increasing in $d$, such that if
$\pi_{k}(d,\eta)$ is false, then for each $i \in [k]$,
\[ \lim_{r \rightarrow \infty} \Pr \left[ \tau_{\rho} = i ~|~
\tau_{\partial T_r}\right] \rightarrow  \frac{1}{k},~~a.a.s. \]
\end{proposition}

For the case of $k = 2$, the exact form of $\pi_2$ is known, $\pi_2(d, \eta) =
\one[d(1-2\eta)^2 > 1]$, which follows from~\cite{EKPS:2000}.
In~\cite{Sly:09a}, the exact threshold is given for $k = 3$, and bounds on the
thresholds are given for $k \ge 5$.  For $k \geq 4$, the exact form $\pi_k$ is
not known, but it holds that if $(1 - k\eta) d < 1$, $\pi_{k}(d, \eta)$ is
false. (This was proved for the case of regular trees in~\cite{Mossel:2001};
the proof for Galton-Watson trees is essentially identical). For all $k$, a
reconstructability threshold in $\eta, d$ provably exists in the limit as
$n\to\infty$; the proof of Proposition \ref{prop:k_threshold} relies on the
monotonicity of $\pi_k$ in $\eta$ and $d$, and the existence of points where
reconstruction is feasible and also points where it is impossible.

The threshold from Proposition \ref{prop:k_threshold} can be translated to an
equivalent threshold $\theta_k(a, b)$ in the stochastic block model.  We show
that even in the labeled stochastic block model (where each node's label is
revealed with probability $p$), if $p$ is small and $\theta_k$ is false then it
is impossible to recover node labels with better accuracy than random guessing.
Specifically, we study the setting where $k$ is fixed, $\theta_k$ is false, and $p\to
0$. We first prove this for the general $k$-cluster case, then give an
alternative proof for the case of two clusters (which results in a more explicit
dependence on $p$). \medskip

\begin{theorem} \label{thm:k_clusters}
Fix $v \in [n]$, and let $(G,R,\sigma) \sim \G(n,k,a,b,p)$, for $a + (k -1 ) b >
k$. Then if the predicate $\theta_k(a, b) = \pi_k(\tfrac{a + (k-1)b}{k},
\tfrac{b}{a + (k-1)b})$ is not satisfied, then for all $i \in \Sigma =
[k]$,
\[
\lim_{p \rightarrow 0} \lim_{n \to\infty} \Pr[\sigma_v = i|G, R, \sigma_R] = \frac{1}{k},~~ a.a.s.
\]
\end{theorem}

The above result says that as the amount of revealed node information goes to
zero, recovering a clustering that is correlated with the true clustering is
not possible if $\theta_k$ is false. The proof of Theorem \ref{thm:k_clusters}
requires some results from the literature which we now state.

We again utilize a coupling between local neighborhoods in $\G(n,k,a,b)$ and a
broadcast process on a rooted Galton-Watson tree. As in
Section~\ref{sec:many_clusters}, let $T$ be a Galton-Watson tree with offspring
distribution $\Poisson(\tfrac{a + (k-1)b}{k})$ and broadcast parameter $\eta =
\tfrac{b}{a + (k-1)b}$.  We fix $v \in [n]$ and let $(G, \sigma) \sim \G(n, k,
a, b)$.  The distribution $(G_r(v), \sigma_{G_r(v)})$ resembles the
distribution $(T_r, \tau_{r})$.

We also use a result of \cite{MNS:2012} which states that conditioned on
$\sigma_{\partial G_r}$, information from further nodes is not helpful in
clustering. \medskip

\begin{lemma}[\cite{MNS:2012}] \label{lem:MNSsep}
Fix $v \in [n]$, and let $(G,R, \sigma) \sim \G(n,k,a,b,p)$, with $a + (k-1)b >
k$. For $r \leq \frac{1}{10 \log( 2(a + (k-1)b))}\log n$, let $C = \{u \in
G~|~d(u,v) > r\}$, $B = \partial G_r$, and $A = \{u \in G~|~d(u,v) \le r\}$.
Then
\[
\Pr[\sigma_{A}~|~\sigma_{B}, \sigma_{C}, G]
= (1 + o(1))\Pr[\sigma_{A}~|~\sigma_{B}, G].
\]
\end{lemma}

In~\cite{MNS:2012}, the lemmas above are stated for the case when $k = 2$;
however, the same proofs apply for any value of $k$.  Armed with
Lemmas~\ref{lemma:tree-coupling},~\ref{lem:MNSsep} and
Proposition~\ref{prop:k_threshold}, we can now prove Theorem~\ref{thm:k_clusters}.

\begin{proof}[Proof of Theorem \ref{thm:k_clusters}]
We begin by proving an analogous result for a broadcast process on a
Galton-Watson tree. Let $T$ be a Galton-Watson tree with average degree $d = (a
+ (k-1)b)/k$. Let $(T, \tau, R) \sim \T( T, k, \eta, p)$, where $\eta =
\frac{b}{a + (k-1)b}$.  Fix some radius $r$ around $\rho$, and let $W_1 = R \cap
T_r$.

Now, we will bound the number of nodes in $W_1$---this will allow us to argue
that as $p \to 0$, $T_r\cap R = \emptyset$.  Let $X_i = |\partial T_i|$; we
argue inductively that $\E[X_i] = d^i$.  Clearly, $X_0 = 1$. For the inductive
step, $\E[X_i~|~X_{i-1}] = d\cdot\E[X_{i-1}]$, and so $\E[|W_1|] = \E[p\sum_{i =
0}^r X_i] = O(pd^r)$.  Applying Markov's Inequality, $\Pr[|W_1| \ge
\tfrac{1}{2}] \le O(pd^r)$.  Let $r = - \frac{1}{2} \log_{d}(p)$, so as $p \to 0$, $r \to
\infty$ and  $\Pr[W_1 \neq \emptyset] \to 0$.

Using this, as $(p, r) \to (0, \infty)$, we have
\begin{equation} \label{eq:tree_equivalence}
\Pr[\tau_v = i~|~ \tau_{W_1}, \tau_{\partial T_r}]
= \Pr[\tau_v = i~|~ \tau_{\partial T_r}]
\qquad a.a.s. ~~\forall i \in [k].
\end{equation}

Since $\theta_k(a, b)$ is false by assumption, $\pi_k(\eta, d)$ is false for the
values of $\eta, d$ given above.  Thus, we can apply Proposition
\ref{prop:k_threshold} to see that,
\begin{equation} \label{eq:nonrec}
\lim_{r \to \infty} \Pr[\tau_v = i~|~\tau_{\partial T_r}]
~=~ \Pr[\tau_v = i] = \frac{1}{k} \qquad \forall i \in [k].
\end{equation}

If we wanted Theorem~\ref{thm:k_clusters} to hold for $T$ rather than $G$ we
would be done. By the Markov property of the broadcast process on $T$,
the information from $\tau_{\partial T_r}$ isolates $\rho$ from the effects of
information beyond $T_r$. However, because we are not in a tree, we must now take
some extra care to apply this conclusion to $G$.

Now, we translate these results to the stochastic block model setting.  We first
apply the coupling in Lemma~\ref{lemma:tree-coupling}~ to
~(\ref{eq:tree_equivalence})~ and ~(\ref{eq:nonrec})---it is clear that the
revealed nodes in $T$ can be coupled with the revealed nodes in $G$.  Let $R_1 =
R \cap G_r(v)$ and let $B = \{u \in G ~|~ d(u,v) = r\}$. Then, we have the
\begin{equation} \label{eq:max_prob}
\lim_{(p, r) \to (0, \infty)} \lim_{n \to \infty} \Pr[\sigma_v = i~|~
\sigma_{R_1}, \sigma_B, G, R] ~=~ \Pr[\sigma_v = i ~|~ \sigma_B, G, R] ~=~
\frac{1}{k} \qquad \forall i \in [k].
\end{equation}

All that now remains is to prove that global information does not help in the
block model setting. To do this, we will look at the entropy of $\sigma_v$
conditioned on different sets of variables.

Using~(\ref{eq:max_prob}) it is clear that in the limit as $n \to \infty$ and
$(p, r) \to (0, \infty)$, $H(\sigma_{v} ~|~ G, R, \sigma_{R_1}, B)$ has the
maximum possible value. By applying Lemma~\ref{lem:MNSsep}, we know that
$H(\sigma_v ~|~ G, R, \sigma_{R_1}, \sigma_{B}) = (1 + o(1)) H(\sigma_{v} ~|~ G,
R, \sigma_{R}, \sigma_{C}, \sigma_{B})$, and hence in the asymptotic limit the
latter conditional entropy is also the maximum possible. Then, by monotonicity
of conditional entropy,
\[ H(\sigma_v ~|~ G, R, \sigma_{R},\sigma_{B}, \sigma_{C}) \leq H( \sigma_{v}
~|~ G, R, \sigma_{R}). \]
Thus, we get that $\lim_{p \to 0} \lim_{n \to \infty} H(\sigma_v ~|~ G, R,
\sigma_{R})$ is the maximum possible. This completes the proof of the theorem.
\end{proof}

\bigskip
%------------------------------------------------------------------
In the special case of $k = 2$ clusters, it is possible to prove the same result
using a slightly different technique. Here, we get a more explicit convergence
rate in terms of $p$. Note that the RHS in the statement of
Theorem~\ref{thm:twoclusters} cannot be smaller than $p$, since with probability
$p$ the node of the label itself is revealed. \medskip

\begin{theorem} \label{thm:twoclusters}
Fix $v \in [n]$, and let $(G,R,\sigma) \sim \G(n,2,a,b,p)$, for $a + b > 2$.
Then if $(a - b)^2 < 2 (a + b)$, then
%%%
\[ \lim_{n \rightarrow \infty} \E \left| \Pr[\sigma_v = 1 ~|~ G, R, \sigma_{R}] -
\frac{1}{2} \right| \leq \frac{1}{2} \sqrt{\frac{p}{1 - \frac{(a -b)^2}{2(a + b)} }} \]
\end{theorem}

For this better dependence, we rely on a result of Evans \etal~\cite{EKPS:2000}
regarding predicting the label of the root, when the labels of some nodes in
the tree are revealed. \medskip

\begin{proposition}[\cite{EKPS:2000}] \label{prop:EKPS} Let $W$ be a finite set of
nodes in the tree $T$. Let $(T, \tau) \sim \T(T, 2, \eta)$ be a labeling of a
tree obtained by the broadcast process as defined in
Section~\ref{sec:model-trees} with alphabet $\Sigma = \{\pm 1\}$,
and parameter $\eta$. Let $S$ be any set of nodes
that separates the root from $W$.  Then,
\[
\bigg(\E\left[ \left|\E[\tau_\rho ~|~ \tau_W] \right| \right] \bigg)^2
\le
2\sum_{v\in S} (1 - 2\eta)^{2d(v,\rho)}
\]
\end{proposition}

\bigskip
\begin{proof}[Proof of Theorem~\ref{thm:twoclusters}.] As in the previous proof,
let $T$ be a Galton-Watson tree with  degree distribution $\Poisson(d)$, for $d
= (a + b)/2$.  For notational convenience, let the set of labels be $\Sigma =
\{\pm 1\}$. Let $(T, \tau, R) \sim \T(T, 2, \tfrac{b}{a + b},p)$. Fix some
radius $r$, and let $W_1 \subseteq T = R \cap T_r$. For any integer $j$,
let $X_j = |W_1 \cap \partial T_j|$. Let $W_2 = \partial T_r = \{v \in T ~|~
d(v, \rho) = r\}$.  Let $W = W_1 \cup W_2$.

We consider the question of predicting the label $\tau_{\rho}$, given all the
labels $\tau_{W}$. Note that when $\theta_2(a,b)$ is false, for the parameters
above $(1 -2\eta)^2 d = (a - b)^2/(2(a + b)) < 1$. Then, using
Proposition~\ref{prop:EKPS}, we have the following:
\begin{align}
\bigg(\E \left|\E[\tau_\rho~|~\tau_W]\right| \bigg)^2
&\leq 2\sum_{v\in W} (1 - 2 \eta)^{2d(v,\rho)}\nonumber \\
&= 2\sum_{v\in W_1} (1 - 2 \eta)^{2d(v,\rho)} + 2\sum_{v\in W_2} (1 - 2 \eta)^{2r}\nonumber \\
&= 2\sum_{j=0}^{r-1} X_j (1 - 2 \eta)^{2j} + 2|\partial T_r|(1 - 2 \eta)^{2r}
\intertext{If we take expectation with respect to the choice of revealed nodes and
the Galton Watson Tree process, since
$\E\left[X_j ~|~ |\partial T_j|\right] = p |\partial T_j|$ and
$\E\left[|\partial T_j|\right] = d^{j}$,}
\E_{T, R}\left[\bigg(\E \left|\E[\tau_\rho~|~\tau_W] \right| \bigg)^2 \right]
&\leq p \left(\sum_{j = 0}^{r-1} (d (1 - 2 \eta)^2)^j
\right) + ((1 - 2 \eta)^2 d)^{r} \nonumber  \\
&\leq \frac{p}{1 - d(1 - 2\eta)^2} + ((1 - 2 \eta)^2 d)^{r} \label{eq:phalf}
\end{align}
Notice that since $(1 - 2\eta)^2d < 1$, $((1 - 2\eta)^2 d)^r \rightarrow 0$ as
$r \rightarrow \infty$.

The rest of the proof proceeds analogously to the proof of
Theorem~\ref{thm:k_clusters} starting at~(\ref{eq:tree_equivalence}) and
applying the Cauchy-Schwarz inequality to~(\ref{eq:phalf}).
\end{proof}

%------------------------------------------------------------

\setcounter{proposition}{0}
\setcounter{conjecture}{0}

\bibliography{all-refs,all}
\bibliographystyle{plain}

\newpage

\appendix
\section{When Little Information Helps}
\label{sec:appendix}
Here, we prove the simple observations described in Section~\ref{sec:intro}
which illustrate the power and limitations of revealed labels in the stochastic
block model.

\subsection{Proof of Proposition~\ref{prop:permutation}}

\begin{proposition} \label{prop:permutation}
Let $C:[n] \to [k]$ be the output of some clustering algorithm with the guarantee
that there exists a permutation $\pi:[k] \to [k]$ such that
\[
\frac{1}{n} \sum_{i} \one[\pi(C(i)) = \sigma_i] \ge \frac{1}{k} + \epsilon,
\]
Then for $p \ge \tfrac{1}{n} \tfrac{512 k}{\epsilon^3} \log\tfrac{4k}{\delta}$,
if a $p$-fraction of node labels are revealed, we can find a function $g:[k] \to
[k]$ such that
\[
\frac{1}{n} \sum_{i} \one[g(C(i)) = \sigma_i] \ge \frac{1}{k} + \frac{\epsilon}{2}
\]
with probability at least $1-\delta$.
\end{proposition} \medskip

The proof follows easily from the following lemma, which is a simple application
of the Chernoff-Hoeffding bound. \medskip

\begin{lemma} \label{lem:plurality}
Let $D$ be a probability distribution over $[k]$, and let $S\sim D^m$ be a
sample. When $m \ge \frac{64}{\epsilon^2}\log(\frac{4k}{\delta})$, for $i =
\mathrm{plurality}(S)$ (ties may be broken arbitrarily), with probability at least $1 - (\delta/2)$,
\[
|D_i - \max_{j} D_{j}| \le \frac{\epsilon}{4},
\]
where $D_j$ is the probability of $j$ under $D$.
\end{lemma}
\begin{proof}
For any $j \in [k]$, let $\hat D_j$ be the fraction of of $j$ in $S$.  By the
Chernoff-Hoeffding bound, $\Pr[|D_j - \hat D_j| \ge \alpha] \le
2\exp(-m\alpha^2)$.  By union bound, the probability that this happens for any
$j \in [k]$ is at most $2k\exp(-m\alpha^2)$. Thus, if we let $m \ge
\frac{1}{\alpha^2}\log(\frac{4k}{\delta})$, this happens with probability at
most $\delta/2$.  Hence, we have $|D_i - \max_{j} D_{j}| \leq 2\alpha$ with
probability at least $1-\delta/2$.  Letting $\alpha = \frac{\epsilon}{8}$
completes the proof.
\end{proof}
\medskip

\begin{proof}[Proof of Proposition~\ref{prop:permutation}]
Let $C:[n] \to [k]$ be a clustering with the assumed property, and let $C_i =
\{v \in [n] ~|~ C(v) = i\}$. If $|C_i| \leq \tfrac{\epsilon n}{4k}$, we assign
each node in $C_i$ a random label.

Let $Y = \{ i ~|~ |C_i| \geq \tfrac{\epsilon n}{4k} \}$. Then for each $i \in
Y$, let $R_i \subseteq C_i$ denote the subset of nodes that are revealed in
$C_i$. Note that $\E[|R_i|] = p |C_i| \geq
\tfrac{128}{\epsilon^2}\log(\tfrac{4k}{\delta})$, for the value of $p$ in the statement
of the proposition.  By a simple Chernoff bound, $\Pr\big[|R_i| < \tfrac{1}{2}
\E[|R_i|]\big] \leq \frac{\delta}{4k}$, whenever $|C_i| \geq \epsilon n/4k$.
Thus, by union bound, for all $i \in Y$, $|R_i| \geq \frac{4k}{\epsilon^2}
\log(\tfrac{64}{\delta})$ except with probability $\delta/2$. We assume that
this is the case for the rest of the proof, allowing the procedure to fail with
probability $\delta/2$.

Now for any $i \in Y$, let $g(i) = \mathrm{plurality}(R_i)$.  By
Lemma~\ref{lem:plurality}, except with probability $\delta/2$, for all $i \in
Y$, $\max_{j \in [k]} \frac{1}{|C_i|} \sum_{v \in C_i} \one( \sigma_v = j) \leq
\frac{1}{|C_i|}\sum_{v \in C_i} \one( g(i) = \sigma_v) + \tfrac{\epsilon}{4}$.
Let $\pi : [k] \rightarrow [k]$ be the optimal permutation given $\langle C_i
\rangle_{i =1}^k$. Then, we have the following,

\begin{align*}
\frac{1}{n} \sum_{v} \one( \pi(C(v)) = \sigma_v) &\leq \frac{1}{n} \sum_{i =
1}^k \max_{j \in [k]} \sum_{v \in C_i} \one( \sigma_{v} = j)  \\
&\leq \frac{1}{n} \sum_{i \in Y} \left( \sum_{v \in C_i} \one(\sigma_v =
g(i)) + |C_i| \frac{\epsilon}{4} \right) + \frac{1}{n} \sum_{i \not\in Y} |C_i| \\
&\leq \frac{1}{n} \sum_{v \in [n]} \one(g(C(i)) = \sigma_v) +
\frac{\epsilon}{4n} \sum_{i \in Y} |C_i|  + \frac{1}{n} \sum_{i \not\in Y} |C_i|
\intertext{Clearly, $\sum_{i \in Y} |C_i| \leq n$ and $\sum_{i \not\in Y} |C_i|
\leq k \cdot \frac{\epsilon n}{4k} \leq \tfrac{\epsilon n}{4}$ . Hence, we have,}
\frac{1}{n} \sum_{v} \one (\pi(C(v) = \sigma_v)) &\leq \frac{1}{n} \sum_{v \in
[n]} \one(g(C(v)) = \sigma_v) + \frac{\epsilon}{2}
\end{align*}

Since by the hypothesis of the proposition, the LHS of the above inequality is
at least $\tfrac{1}{k} + \epsilon$, the assertion holds.
\end{proof}

%----------------------------------------------

\subsection{Proof of Proposition~\ref{prop:local_sucks}}
\label{app:local}

Now, we discuss the impact of revealed labels in the context of local
algorithms. We use the definition of local algorithms as in
\cite{GamarnikSudan:14}. (The reader is referred to their paper and references
therein for more background on local algorithms.) \medskip

\begin{definition}
Let $G$ be a graph with node set $V$, and for each $v \in V$, let $X_v \in [0,
1]$ uniformly at random. An {\em $r$-local algorithm} on $G$ is one in which the
value of each node $v \in V$ is decided by a function $f_v(G_r(v), X_r(v))$,
where $X_r(v)$ is the set of samples from $D$ associated with $G_r(v)$.
\end{definition}

Here, we justify the intuitive statement that no $r$-local algorithm can
accurately reconstruct clusters in the unlabeled stochastic block model for $r
= o(\log n)$. \medskip

\begin{proposition} \label{prop:local_sucks}
In the unlabeled stochastic block model, let $A$ be a local algorithm with
node functions $\{f_v\}:G_r(v) \to \Sigma$, where here $G_r(v)$ denotes the
structural information and random variables on the neighborhood of radius $r =
o(\log n)$ around $v$. Then for all $\eps > 0$,
\[
\lim_{n \to \infty} \Pr_{G, X}[\max_{\pi} \frac{1}{n} \sum_v 1(f_v(G_r(v)) = \pi(\sigma_v)) \geq \frac{1}{k} + \eps] = 0,
\]
where the maximum is taken over all possible permutations of the labels.
\end{proposition}

\begin{proof}
By the union bound over all $k!$ permutations it suffices to show that for each fixed permutation $\pi$:
\[
\lim_{n \to \infty} \Pr_{G, X}[\frac{1}{n} \sum_v 1(f_v(G_r(v)) = \pi(\sigma_v)) \geq \frac{1}{k} + \eps] = 0
\]
Without loss of generality, we may assume that $\pi$ is the identity permutation.
Let
\[
Z = \frac{1}{n} \sum_v 1(f_v(G_r(v)) = \sigma_v)
\]
Note that for each $v$, $\sigma_v$ is distributed uniformly conditioned on
$f_v(G_r(v))$  and therefore $\E[Z] = 1/k$.  Thus, in order to prove the claim it
suffices, by Chebychev's Inequality to show that $\Var[Z]$ is $o(1)$ or
equivalently that $\E[Z^2] = 1/k^2 + o(1)$.  Now
\begin{align*}
\E[Z^2] &= \frac{1}{k n} + \frac{1}{n^2} \sum_{v \neq u} \Pr[f_v(G_r(v)) =
\sigma_v, f_u(G_r(u)) = \sigma_u] \\
&= \frac{1}{k n} + \frac{1}{k n^2} \sum_{v
\neq u} \Pr[f_v(G_r(v)) = \sigma_v | f_u(G_r(u)) = \sigma_u]
\end{align*}
Thus the proof reduces to showing that for a fixed $u \neq v$ (chosen before the
graph is labeled and the edges are generated) it holds that
\[
\Pr[f_v(G_r(v)) = \sigma_v | f_u(G_r(u)) = \sigma_u] = \frac{1}{k} + o(1).
\]
Now:
$\Pr[f_v(G_r(v)) = \sigma_v | f_u(G_r(u)) = \sigma_u]$ is bounded by
\[
\Pr[f_v(G_r(v)) = \sigma_v | f_u(G_r(u)) = \sigma_u, d(u,v) > 2r]
+
\Pr[d(u,v) \leq 2 r | f_u(G_r(u)) = \sigma_u]
\]
\[
\leq \Pr[f_v(G_r(v)) = \sigma_v | f_u(G_r(u)) = \sigma_u, d(u,v) > 2r]  + o(1),
\]
since with high probability $u$ and $v$ are at distance $\Omega(\log n)$.

If there are $\gamma_i$ nodes with label $i$ in $G_{2r}(u)$, the distribution
of $\sigma_v$ for a random $v$ with $d(u,v) > 2r$ has total variation distance
at most $\tfrac{2k}{n- |G_{2r}(u)|}\sum_{i \in [k]}\gamma_i$ from uniform;
knowing $u$ was assigned $\sigma_u$ and $d(u,v) > 2r$ only yields information
about the distribution of values of $\gamma_i$, and has no other implications
for $v$.  Clearly, $\sum_{i \in [k]}\gamma_i = |G_{2r}(u)|$, and with high
probability, $|G_{r2}(u)| = O(d^{2r} \log n)$ for $d = \tfrac{a + (k-1)b}{k}$.
Thus, as $n\to\infty$, $\Pr[f_v(G_r(v)) = \sigma_v ~|~ f_u(G_r(u)) = \sigma_u,
d(u,v) > 2r] = \frac{1}{k}$ completing the proof.

\end{proof}
\bigskip

%----------------------------------------------

\subsection{Proof of Proposition \ref{prop:local_with_info}}

Before giving a formal statement and proof of
Proposition~\ref{prop:local_with_info}, we need to introduce some notation
related to broadcast processes on trees. Let $(T, \tau) \sim \T(T, 2, \eta)$,
where $T$ is a Galton-Watson tree with offspring distribution $\Poisson(d)$. Let
\[ \TR^*(d, \eta) = \lim_{r \rightarrow \infty} \E\left| \Pr[\tau_{\rho} = 1 ~|~
\tau_{\partial T_r}] - \frac{1}{2} \right| \]
It follows from the work of Evans~\etal that $\TR^*(d,\eta) >0$
if and only if $d (1 - 2\eta)^2 > 1$~\cite{EKPS:2000}.

Mossel~\etal~\cite{MNS:2013a} looked at the robust reconstruction problem on
trees. Let $(T, \tau) \sim \T(T, 2, \eta)$ be as defined above. For some
parameter $\delta \in [0, 1/2)$, let $\tilde{\tau}_u$ be the random variable,
such that $\tilde{\tau}_u = \tau_u$ with probability $1 - \delta$, and
$\tilde{\tau}_u = 1-\tau_{u}$ with probability $\delta$.  In~\cite{MNS:2013a},
the authors consider the question of reconstruction of the root label given the
noisy labels, $\tilde{\tau}_{\partial T_r}$,  in the limit as $r \to \infty$.
They showed that if
\[ \widetilde{\TR}^*(d, \eta) = \lim_{r \rightarrow \infty} \E\left| \Pr[\tau_{\rho}
=1 ~|~ \tilde{\tau}_{\partial T_r}] - \frac{1}{2} \right|, \]
then for any $\delta \in [0, 1/2)$, whenever $d (1 - 2\eta)^2 \geq C$ for a
sufficiently large constant $C$, $\widetilde{\TR}^*(d,\eta) = \TR^*(d, \eta)$.
\medskip

\begin{proposition} \label{prop:local_with_info}
Let $(G, R, \sigma_{R}) \sim \G(n, 2, a, b, p)$, with $a + b > 2$. Then, there
exists a large constant $C$, such that if $(a -b)^2 > C(a + b)$, there is a
local algorithm $A$ such that if $A(v)$ denotes the label output by the
algorithm, for a random node $v$,
\[ \lim_{p \to 0} \lim_{n \to \infty} \Pr[ A(v) = \sigma_v] = \frac{1}{2} +
\TR^*(\tfrac{a + b}{2}, \tfrac{b}{a + b}) \]
\end{proposition}
\begin{proof}
We consider the corresponding question on trees. Let $d = \tfrac{a + b}{2}$ and
$\eta = \tfrac{b}{a + b}$. Let $T$ be a Galton-Watson tree with offspring
distribution $\Poisson(d)$. Let $(T, \tau, R) \sim \T(T, 2, \eta, p)$, and let
$R_r = \{v \in R ~|~ d(\rho, v) \le r\}$ for some $r(p)$ such that
$r \to \infty$ as $p \to 0$. Our goal is to show
that whenever $d (1 - 2\eta)^2 > C$, where $C$ is the constant in the work
of~\cite{MNS:2013a},
\begin{align}
\lim_{(p,r) \rightarrow (0,\infty)} \E \left| \Pr[\tau_{\rho} = 1 ~|~ R_r, \tau_{R_r}]  - \frac{1}{2}
\right| &= \TR^*(d, \eta)  \label{eqn:prop3-goal}
\end{align}
To show~(\ref{eqn:prop3-goal}), fix some radius $r$, then notice that by the
monotonicity of conditional variances, $\Var(\tau_{\rho}~|~ R_r, \tau_{R_r},
\tau_{\partial T_r}) \leq \Var(\tau_{\rho} ~|~ R_r, \tau_{R_r})$. Consider
$\E[|R_r|]$. An easy calculation (see the proof of Theorem~\ref{thm:k_clusters}
for details), shows that $\E[|R_r|] = O(pd^r)$, thus when $r = -\tfrac{1}{2}
\log_d(p)$, the probability that $R_r \neq \emptyset$ goes to $0$ as $p
\rightarrow 0$ by Markov's inequality. Conditioning on the event that this is
indeed the case, $\Pr[\tau_{\rho} =1 ~|~ \tau_{\partial T_r}, \tau_{R_r}, R_r]
=\Pr [\tau_{\rho} = 1~|~\tau_{\partial T_r}]$.  Thus, we have,
\begin{align}
\lim_{(p, r) \rightarrow (0,\infty)} \Pr[\tau_{\rho} = 1~|~\tau_{\partial T_r},
\tau_{R_r}, R_r] = \lim_{r\rightarrow \infty} \Pr[\tau_{\rho} = 1~|~\tau_{\partial
T_r}]
\nonumber
\intertext{Using the above equation together with the fact that
$\Var(\tau_{\rho} ~|~ \tau_{\partial T_r},\tau_{R_r}, R_r)
\leq \Var(\tau_{\rho} ~|~
\tau_{R_r}, R_r)$, we get that,}
\lim_{p \rightarrow 0} \E\left| \Pr[\tau_{\rho}  = 1~|~ \tau_{R_r}, R_r] - \frac{1}{2}
\right| \leq \TR^*(d, \eta) \label{eqn:prop3-upperbound}
\end{align}

For the other direction, let $p > 0$ and fix some radius $r$. For $u \in
\partial T_r$ define the random variable $\tau^\prime_u = \tau_u$ if $u \in R$,
and $\tau^\prime_u \in \{0, 1\}$ uniformly at random if $u \not \in R$. Note
that conditioned on $\tau_{\rho}$, the random variables $\langle \tau^\prime_u
\rangle_{u \in \partial T_r}$ and the noisy labels, $\langle \tilde{\tau}_u
\rangle_{u \in \partial T_r}$ are identically distributed if $\delta =
\tfrac{1}{2} - \tfrac{p}{2}$. Again, we have that, $\Var(\tau_{\rho} ~|~ R_r,
\tau_{R_r}) = \Var(\tau_{\rho} ~|~ R_r, \tau_{R_r}, \tau^\prime_{\partial T_r \setminus
R_r}) \leq \Var(\tau_{\rho} ~|~ \tau^\prime_{\partial T_r})$, where the first
equality holds since $\tau^\prime_u$ for $u \not \in R$ is independent of
$\tau_{\rho}$ and the inequality holds by monotonicity of conditional variances.
By definition,
\begin{align}
\lim_{r \rightarrow \infty} \E \left| \Pr[\tau_{\rho} = 1 ~|~
\tau^\prime_{\partial T_r}] - \frac{1}{2} \right| = \widetilde{\TR}^*(d,
\eta) \nonumber
\intertext{ Using the above equation together with the relationships between the
variances, we have}
\lim_{p \rightarrow 0} \E\left| \Pr[\tau_{\rho} =1 ~|~ \tau_{R_r}, R_r] -
\frac{1}{2} \right| \geq \widetilde{\TR}^*(d, \eta) \label{eqn:prop3-lowerbound}
\end{align}

Combining~(\ref{eqn:prop3-upperbound}) and~(\ref{eqn:prop3-lowerbound}) together
with the result in~\cite{MNS:2013a}, we have that whenever $d (1 -
2\eta)^2 \geq C$,~(\ref{eqn:prop3-goal}) is true.

Finally, the mapping from the result on trees to the block model follows from a
coupling between local neighborhoods of nodes in the block model with the
broadcast process on trees. For details see Lemma~\ref{lemma:tree-coupling} and
its application in the proof of Theorem~\ref{thm:k_clusters}.

This implies the proposition, as we can take $A$ to be the Belief Propagation
algorithm (see \eg \cite{MNS:2013a}) with radius $r$, with nodes in $R$
initialized according to their labels and with nodes outside of $R$
initialized randomly. Note that belief propagation is known to converge on
trees.
\end{proof}

\subsection{Proof of Proposition~\ref{prop:local_count}}
\label{sec:recon}

Given an instance of the stochastic block model $(G, \sigma, R) \sim \G(n, k,
a, b, p)$ and the corresponding Galton-Watson tree and broadcast process $(T,
\tau, R) \sim \T(T, k, \eta, p)$, we now prove that if $d \lambda^2 =
(a-b)^2/(k(a + (k-1)b)) > 1$, the plurality of labels at distance $\ell$ from a
node $v$ provides a robust way to recover a $v$'s label for every information
$p$.  The argument is based on the reconstruction argument for the label of a
root in a broadcast process on trees, and the fact that the application of the
second moment method in this argument is robust to noise in the leaf labels.
This was implicit in \cite{MosselPeres:03} and more explicit
in~\cite{MNS:2013a}.  Interestingly, the proof will show that in the case of
Poisson Galton-Watson tree, a simple plurality style rule is sufficient for
reconstruction.

\medskip

\begin{proposition}
\label{prop:local_count}
Let $(G, \sigma, R) \sim \G(n, k, a, b, p)$, with $a + (k-1)b > k$.  Then,
there exists a constant $\epsilon = \epsilon(a,b,k,p)$, such that if $(a-b)^2 >
k(a + (k-1)b)$, there is a local algorithm $A$ such that if $A(v)$ denotes the
label output by the algorithm, for a random node $v$,
\[ 
\Pr[A(v) = \sigma_v] \ge \frac{1}{k} + \epsilon.
\]
The result also holds for the noisy-label model.
\end{proposition}

\begin{proof}
We consider the corresponding question of root reconstruction in a tree.  Let
$T$ be a Galton-Watson tree with offspring distribution $\Poisson(d)$ for $d =
\tfrac{a + (k-1)b}{k}$, and let $(T, \tau, R) \sim \T(T, k, \eta, p)$ for $\eta
= \tfrac{b}{a + (k-1)b}$. Consider the broadcast process on $T$.  

One representation of the broadcast process is that along each edge, each
symbol is copied probability $\lambda = 1- k \eta$ and is otherwise randomized
to one of the $k$ symbols.  Recall that the second eigenvalue of the broadcast
matrix in this case is given by $\lambda$.  Fix a level $\ell$ of the tree.
For any $v \in \partial T_{\ell}$, let $\tilde{\tau}_v = \tau_v$ with
probability $\mu = 1 - k\delta$ and is chosen randomly from $[k]$ otherwise. We
will assume that our algorithm has access to $\tilde{\tau}_{\partial T_\ell}$;
given $R$ and $\tau_R$, $\tilde{\tau}$ can be constructed by letting
$\tilde{\tau}_v = \tau_v$ if $v \in R$, and choosing randomly otherwise.  This
gives $\mu = p$ and $\delta = \tfrac{1-p}{k}$.

For each leaf node $v \in \partial T_{\ell}$, let $X^v$ denote a random vector
where $X^v_i = \one(\tilde{\tau}_v = i) - \tfrac{1}{k}$. Let $S_{\ell} =
\sum_{v \in \partial T_{\ell}} X^v$--in words, $S_\ell$ is a vector whose
positive entry is the plurality of label colors at level $\ell$.

Note that when the color of the root is chosen uniformly at random,
$\E[S_{\ell}] =0$.  Let $\E^i$ denote expectations conditioned on $\tau_{\rho}
= i$. Then, note that 
\[ 
\E^i[S_{\ell}] = \mu (d\lambda)^{\ell} (e_i - \tfrac{1}{k}1_k), 
\] 
where $e_i$ denotes the unit vector with $1$ in the $i$th co-ordinate and $1_k$
is the all-ones vector. This follows from the fact that $\E^i[X^v] =
\mu\lambda^\ell (e_i - \tfrac{1}{k} 1_k)$, since $(e_i - \tfrac{1}{k}1_k)$ is
an eigenvector of the broadcast matrix with eigenvalue $\lambda$.

We would like to bound $\E^i[S_\ell^2]$ in order to apply the second moment method.

We control the second moment by induction.  Let 
\[
	B_{j} = \max_{i \in [k]} |\E^i[({S_j}_i) ({S'_j}_i)]| %\max \sum_{i=1}^q |\E[ X_{\ell}(i) Y_{\ell}(i)]| 
\]
where $S_j$ and $S'_j$ are the sums corresponding to two sibling sub-trees of
$j$ levels each (so that the root labels of the trees of $S_j$ and $S'_j$ are
correlated).  Similarly, let 
 \[
	 A_{j} = \max_{i\in[k]} \E^i[({S_{j}}_i)^2].
\]

We obtain a recurrence for the value of $A_j$ by considering the contribution
from subtrees rooted at the root's children (where we have applied the triangle
inequality):
\[
A_{j} ~\leq~ \E[D(D-1)] B_{j-1} + \E[D] A_{j-1}, 
\]
and 
\[
B_{j} ~\leq~ \E[D]^2 \lambda^2 B_{j-1},
\]
where $D$ is a random variable corresponding to the degree of the root.  We can
bound the initial values of the recurrence by: 
\[
B_0 \leq \mu^2 \lambda^2, \quad A_0 = 1.
\]
Since $\E[D] = d$, it is easy to solve for $B_j$ and get 
\[
B_{j} ~\leq~ \mu^2 d^{2 j} \lambda^{2 j + 2}
\]
Plugging this back to $A_{j}$ and using the fact that the variance and expected
value of a Poisson variable are identical, we get that 
\[
A_{j} ~\leq~ d^2 \mu^2 d^{2 j - 2} \lambda^{2 j} + d A_{j-1} 
~=~ 
\mu^2 d^{2 j} \lambda^{2 j} + d A_{j-1} 
\]
which then implies: 
\[
A_{j} ~\leq~ \mu^2 d^{j} (\sum_{i=1}^{j} (d \lambda^2)^{i}) + d^{j}.
\]
Since $d \lambda^2 > 1$ the expression above is bounded by 
\[
A_{j} ~\leq~ C \mu^2 d^{j} (d \lambda^2)^{j} + d^{j},
\] 
for some absolute constant $C = C(d\lambda^2)$. Thus if we look at the
difference of means: 
\[
	\E^i[{S_\ell}_i] - \E^j[{S_\ell}_i] ~=~ \mu(d \lambda)^{\ell},
\]
and the second moment is bounded above by
\[
	\E[{S_\ell}_i^2] ~\le~ C \mu^2 d^{\ell} (d \lambda^2)^{\ell} + d^{\ell},
\]
irrespective of the label of the root.  Thus as $\ell \to \infty$, the ratio
between the square of the first moment and the second moment is bounded below
by $\tfrac{1}{C}$ for every value of $\mu \neq 0$.  Thus by the standard
application of the second moment method (see \eg Proposition 7.8
in~\cite{LPW:2006}), the label of $\rho$ is
reconstructable with probability at least $\tfrac{1}{k} + \epsilon$, for some
constant $\epsilon = \epsilon(C)$. The
proof follows by applying the coupling from Lemma ~\ref{lemma:tree-coupling}.
\end{proof}

\section{Conjecture}
\label{sec:expts}
\subsection{The Uselessness of Global Information}

In the case of two clusters, we conjecture that whenever any node label
information is present, a local algorithm is already able to recover the
clusters optimally.  The algorithm is the following: Fix some radius $r$, for
each $v \in G$, look at the neighborhood $G_r(v)$, let $R_r \subseteq G_r(v)$
denote the revealed nodes in the neighborhood. As long as $r \leq c \log(n)$
for a sufficiently small constant $c$, the neighborhood is a tree with high
probability. Then $\Pr[\sigma_v = 1 ~|~ R_r, \sigma_{R_r}]$ can be computed
exactly by belief propagation. We conjecture that this is optimal. This would
follow from a related conjecture regarding the broadcast process on trees and an
application of Lemma~\ref{lemma:tree-coupling}. \medskip

\begin{conjecture}
Let $T$ be infinite tree with root $\rho$. Let $(T, \tau, R) \sim \T(T, 2, \eta,
p)$ (see Section~\ref{sec:model}). Then for any $p > 0$ and $\eta < 1/2$,
\[
\lim_{r \to \infty} \E\big| \Pr[\tau_{\rho = 1} ~|~ \tau_{R}] - \Pr[\tau_{\rho
= 1} ~|~ \tau_{R}, \tau_{\partial T_r}] \big| = 0.
\]
\end{conjecture}

\subsection{Simulation}

To test this conjecture, we ran the Belief Propagation algorithm on $3$-regular
trees of depth $10$, in which labels were assigned to nodes according to
broadcast processes starting at the root. Let $L$ denote the set of leaves at
level $10$. Each node in the interior was revealed independently with
probability $p$, to get the set $R$. We considered $p \in \{ 0.01, 0.05, 0.10,
0.20\}$. We also tried various settings of the broadcast parameter, $\eta$. We
chose $\eta \in \{0.1,\eta_c , 0.3, 0.4\}$, where $\eta_c = \frac{1}{2}\left(1 -
\frac{1}{\sqrt{3}}\right)$ is the threshold value for the setting considered.

The labeling process was always initiated with the root having label $1$. Thus,
we were interested in the posterior probability of the root being labeled $1$ in
various cases.  We computed this posterior probability in three cases: (i) using
only the labels at the leaves, denoted by $p_L$ (ii) using only the interior
nodes, denoted $p_R$, and (iii) using both the leaves and the interior nodes,
denoted by $p_{L, R}$.

In the first case, only global information is used---\ie the set of labels at the
boundary is the maximum possible information that can be inferred using the
global properties of the graph. Thus, in some sense this is an upper bound on
the utility of global information.  In the second case, only local
information in the form revealed nodes in the neighborhood is used. Finally, in
the the third case, both local and global information is used.

Our conjecture suggests that as $r \to \infty$, $|p_{R, L} - p_{R}| \rightarrow
0$. Figure \ref{fig:localvglobal} shows our results. Each plot corresponds to
a fixed value of $\eta$, and displays the average distance
$|p_{R,L} - p_R|$ for different values of $p$. We ran the simulation multiple
times for each setting of $p$ and $\eta$ and the standard deviation is marked on
the plot.

\begin{figure}
\begin{center}
\includegraphics[width=0.9\textwidth]{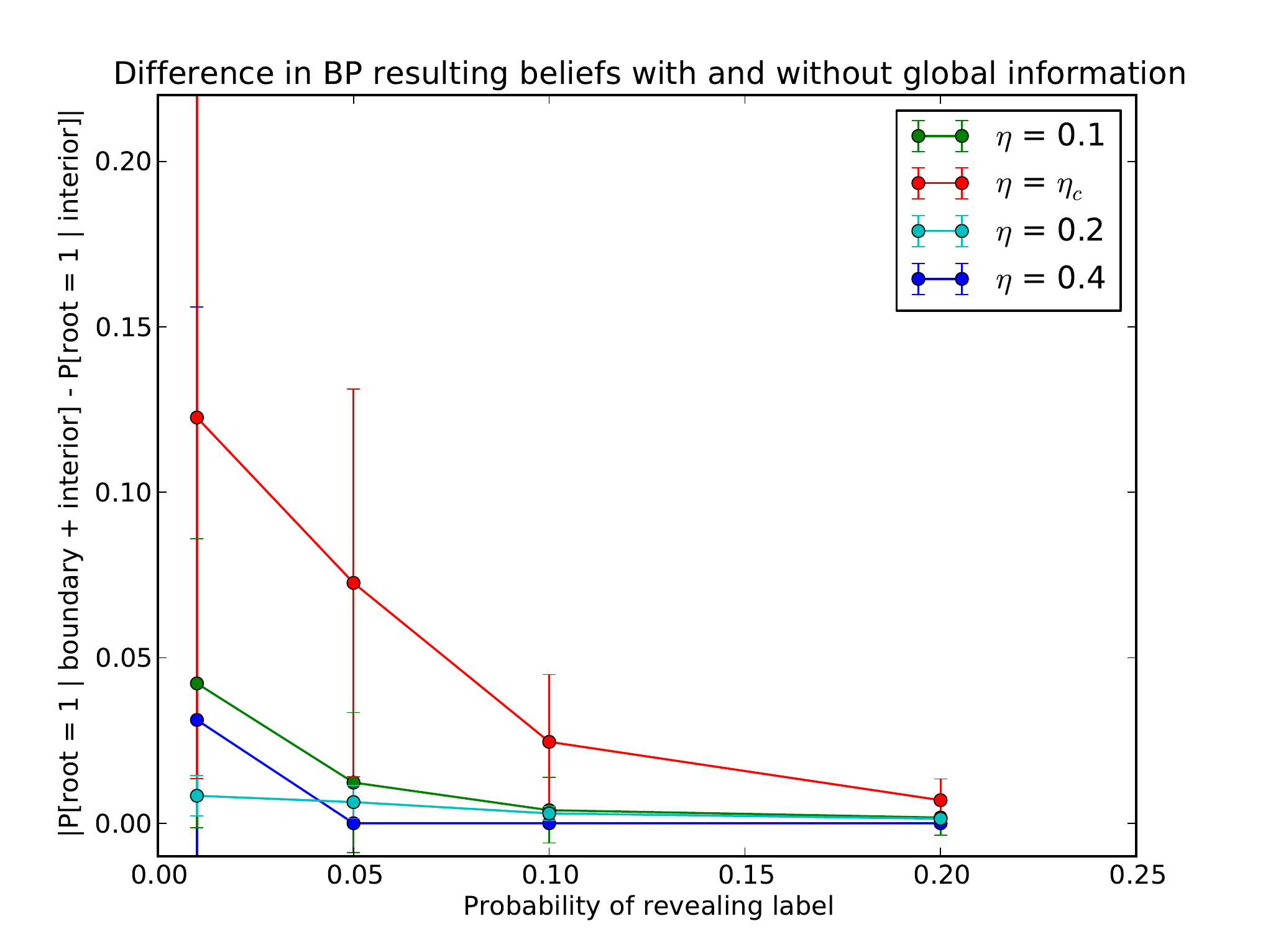}
\caption{The average distance $|p_{R,L} - p_R|$ is shown for
$\eta = 0.1, \eta_c, 0.3, 0.4$ and $p = 0.01, 0.05, 0.1, 0.2$.}
\label{fig:localvglobal}
\end{center}
\end{figure}

\end{document}